\documentclass[a4paper,reqno]{amsart}

\usepackage{latexsym}
\usepackage[english]{babel}
\usepackage{fancyhdr}
\usepackage[mathscr]{eucal}
\usepackage{amsmath}
\usepackage{mathrsfs}
\usepackage{mathabx}
\usepackage{float}
\usepackage{amsthm}
\usepackage{amsfonts}
\usepackage{amssymb}
\usepackage{amscd}
\usepackage{bbm}
\usepackage{graphicx}
\usepackage{graphics}
\usepackage{latexsym}
\usepackage{color}
\usepackage{subfig}
\usepackage{hyperref}
\usepackage{bbm}

\usepackage{scalerel,stackengine}
\stackMath
\newcommand\reallywidehat[1]{%
\savestack{\tmpbox}{\stretchto{%
  \scaleto{%
    \scalerel*[\widthof{\ensuremath{#1}}]{\kern-.6pt\bigwedge\kern-.6pt}%
    {\rule[-\textheight/2]{1ex}{\textheight}}
  }{\textheight}%
}{0.5ex}}%
\stackon[1pt]{#1}{\tmpbox}%
}

\newcommand{\E}{\mathbb E}

\theoremstyle{plain}
\newtheorem{theorem}{Theorem}[section]
\newtheorem{lemma}[theorem]{Lemma}

\newtheorem{proposition}[theorem]{Proposition}

\theoremstyle{definition}

\newtheorem{remark}[theorem]{Remark}
\newtheorem*{remark*}{Remark}

\numberwithin{equation}{section}

\begin{document}

\title[Cutoff for a class of auto-regressive models]{Cutoff for a class of auto-regressive models with vanishing additive noise}

\author[B.~Gerencs\'er]{Bal\'azs Gerencs\'er}
\address[B.~Gerencs\'er]{Alfréd Rényi Institute of Mathematics \\ Reáltanoda utca 13-15, Budapest 1053 (HU) and E\"otv\"os Lor\'and University, Department of Probability and Statistics,  P\'azm\'any P\'eter s\'et\'any 1/C, Budapest 1117 (HU)}
\email{gerencser.balazs@renyi.hu}
\author[A.~Ottolini]{Andrea Ottolini}
\address[A.~Ottolini]{Department of Mathematics, University of Washington \\ Seattle WA 98102 (USA).}
\email{ottolini@uw.edu}

\begin{abstract}
We analyze the convergence rates for a family of auto-regressive Markov chains $(X^{(n)}_k)_{k\geq 0}$ on $\mathbb R^d$, where at each step a randomly chosen coordinate is replaced by a noisy damped weighted average of the others. The interest in the model comes from the connection with a certain Bayesian scheme introduced by de Finetti in the analysis of partially exchangeable data. Our main result shows that, when $n$ gets large (corresponding to a vanishing noise), a cutoff phenomenon occurs. 
\end{abstract}

\date{\today}

\subjclass[]{}

\thanks{}
\maketitle

\section{Introduction}
Markov chains are used on a daily basis to sample from intractable distributions \cite{1995}. Under suitable ergodicity assumptions one is guaranteed that, after many iterations, a sample from the chain resembles that of its stationary distribution. For both practitioners and theoreticians, a natural question is to understand what ``many" and ``resemble" mean in this context.\\ \\ Our interest will be in a class of measures on $\mathbb R^d$ for some $d\geq 2$. A classical way to approach the problem goes as follows: if $\pi_k$ denotes the law of the Markov chain after $k$ steps, and $\pi$ its stationary measure, one is trying to understand how the total variation distance to stationarity
\begin{align*}
    d_{tv}(\pi_k,\pi):=\sup_{E\subset \mathbb R^d}|\pi_k(E)-\pi(E)|=\inf_{X_k\sim \pi_k, X\sim \pi}\mathbb P(X_k\neq X),
\end{align*}
varies as $k$ increases, the last equality being the well-known coupling interpretation of total variation distance. In the display above, the supremum is taken over all Borel sets while the infimum is taken over all couplings of $\pi$ and $\pi_k$.\\ \\ Often, the evolution of the chain depends on an additional parameter $n$, and it becomes important to understand what is the right sequence $k=k(n)$ at which the transition to randomness occurs, i.e., the total variation distance drops as needed. Typically, the parameter $n$ is related to the size of the state space, though it could encode something different. In our case, it will be related to the magnitude of the noise. For a friendly introduction to the slew of techniques and results on the subject we refer the reader to \cite{LevinPeresWilmer2006}. \\ \\ 
We consider Markov chains $(X^{(n)}_k)_{k\geq 0}$ on $\mathbb R^d$ that updates coordinates one at a time according to the auto-regressive scheme \eqref{regressive}. The regime of interest is that a small additive noise, corresponding to $n$ getting large. Informally, our result is that under mild assumptions, the chain takes order $\log n$ steps to mix. Moreover, we also prove that the transition to randomness occurs in a window of size $\sqrt{\log n}$. This is referred to as the \emph{cutoff phenomenon} \cite{Diaconis1996}. We also determine the location of the cutoff -- i.e., the constant factor of the leading term $\log n$ -- which is closely related to the convergence of a certain auxiliary Markov chain on the unit sphere.\\ \\
We now proceed with a formal definition of the model.
\subsection{The setup}
Given $d\geq 2$, let $P=(p_{ij})_{1\leq i,j\leq d}$ be the transition probabilities of a connected network without loops. For $x\in\mathbb R^d$, we define $\hat x$ by
\begin{equation}\label{averagingdef}
    \hat x_i:=(Px)_i=\sum_{j=1}^dp_{ij}x_j,
\end{equation}
where the sum actually runs over $j\neq i$ owing to the assumption that the network has no loops.
Given $e_1,\ldots, e_d\in (0,1)$, define $A_i:\mathbb R^d\rightarrow \mathbb R^d$, for $1\leq i\leq d$, by setting
\begin{equation}\label{contractionpart}
(A_ix)_j=x_j,\, (j\neq i), \quad (A_ix)_i=e_i\hat x_i.
\end{equation}
Also, given $\sigma_1,\ldots, \sigma_d \in (0,\infty)$, define $b_i:\mathbb R\rightarrow \mathbb R^d$ by
\begin{align*}
    (b_i(z))_j=0, \,(j\neq i)\quad (b_i(z))_i=\sigma_i z.
\end{align*}
Let $U$ denote the uniform measure on $\{1,\ldots ,d\}$, and let $\gamma$ be an absolutely continuous probability measure on $\mathbb R$ with $\int\max(\log x,0)\gamma(dx)<\infty$.
Given $X_0\in\mathbb R^d$ and independent random variables $I_1, Z_1, I_2, Z_2, \ldots$ where the $I_i$'s are distributed according to $U$, while  the $Z_i$'s are distributed according to $\gamma$, define now a family of Markov chains on $\mathbb R^d$ via
\begin{equation}\label{regressive}
    X^{(n)}_k=A_{I_{k}}X^{(n)}_{k-1}+\frac{1}{n}b_{I_k}(Z_k).
\end{equation}
In words, at each step a randomly chosen coordinate is replaced by a damped weighted average of the others, to which a (small) noise is added.
Owing to the assumption of $\gamma$, we are guaranteed by Theorem 2.1 in \cite{Diaconis1999} that $X^{(n)}_k$ has a unique stationary distribution (for fixed $n$), which is the law of $\overline X^{(n)}$ defined in terms of the backward iteration
\begin{equation}\label{stationary}
    \overline X^{(n)}=\frac 1 nb_{I_1}(Z_1)+\frac 1 n A_{I_1}b_{I_2}(Z_{2})+\frac 1 nA_{I_1}A_{I_2}b_{I_3}(Z_{3})+\ldots
\end{equation}
Our main goal is to analyze the rate of convergence to stationarity for a large class of initial data. 
\begin{theorem}\label{maintheorem}
Let $\pi^{(n)}_k$ and $\pi^{(n)}$ be the laws of $X^{(n)}_k$ and $\overline X^{(n)}$ as defined above. Consider any initial condition $X_0\in [a_-, a_+]^d$ for some fixed $0<a_-\leq a_+$. Then, there exists a constant $\alpha\in (-\infty,0)$ independent of $n$ such that if 
\begin{equation}\label{choiceofk}
    k=k(n,\beta):=\left\lfloor\frac{\log n+\beta\sqrt{\log n}}{-\alpha}\right\rfloor, 
\end{equation}
then we have
\begin{align*}
  \lim_{\beta\rightarrow -\infty}\lim_{n\rightarrow +\infty}d_{tv}(\pi, \pi_k)=1, \quad \lim_{\beta\rightarrow +\infty}\lim_{n\rightarrow +\infty}d_{tv}(\pi, \pi_k)=0,
\end{align*}
uniformly over $X_0$.
\end{theorem}
\begin{remark}The constant $\alpha$ is defined in terms of  a certain auxiliary Markov chain on the unit sphere (see \eqref{alpha}), though its explicit value is inaccessible in general. However, using the bound 
\begin{align*}
\|\mathbb E[A_I]\|\leq 1-\frac{1}{d}\left(1-\max e_i\right)
\end{align*}
and some easy convexity argument, one can deduce from \eqref{alpha} the lower bound
\begin{align*}
-\alpha\geq \frac{1-\max e_i}{d}.
\end{align*}
The fact that the bound deteriorates as $d$ grows matches the intuition that the mixing time of the Gibbs sampler increases with the dimension. As for the dependence on the $e_i$s, notice that if they are all equal to one then the chain may not admit a stationary distribution. 
\end{remark}
\begin{remark}
As it will be clear from the proof, the second conclusion of the theorem -- i.e. the limit as $\beta\rightarrow +\infty$ -- holds even for sequences with some coordinates being equal to zero. On the other hand, the first conclusion does not hold in the case $X_0=(1,0,\ldots, 0)$ since with positive probability (namely, if the first coordinate is selected first) the chain will mix in a bounded number of steps. If $X_0$ has non-negative coordinates with at least two of them being strictly positive, it is easy to show that with high probability all coordinates will be bounded away from zero in a bounded number of steps, and thus our result applies.
\end{remark}
Let us give an overview of the main heuristic behind the proof.
We start by analyzing the chain that we obtain by averaging over the randomness stemming from the $Z_k$'s. The core of the proof is to show that this chain is $O(1/n)$ with high probability precisely when $k$ is given by \eqref{choiceofk} for some fixed $\beta$. Then, a standard machinery (namely, the concentration properties of the stationary distribution and the absolute continuity of $\gamma$) allows us to conclude. 
\\ \\
In the case $d=2$, re-selecting the same coordinate has no effect on the distribution of the Markov chain, so that one can think of choosing coordinates in a deterministic fashion. Moreover, this allows for an explicit evaluation of $\alpha$. If $\gamma$ is the law of a normal random variable, this allows the numerical estimation of the total variation distance, displayed in Figure \ref{fig:hitallrange}, in striking accordance with the theoretical results. 
\begin{figure}[h]
    \centering
        \includegraphics[width=0.6\textwidth]{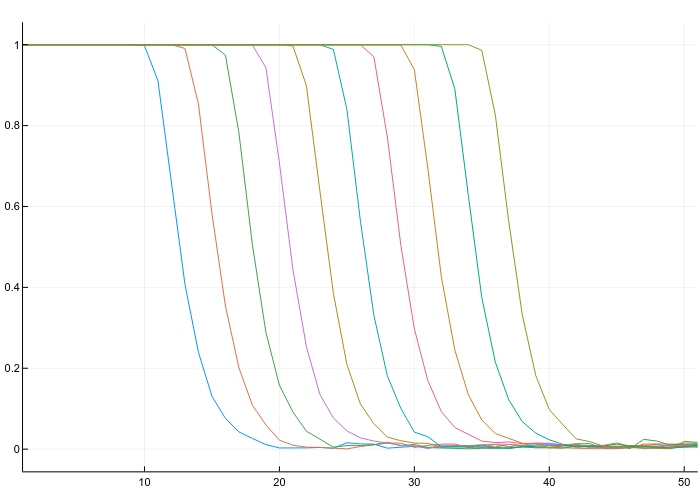}
    \caption{Numerical estimation of the total variation distance between $\pi^{(n)}_k$ and $\pi^{(n)}$, in the case $d=2$, $e_1=e_2=0.55$, and $\gamma$ being a standard normal random variable. This results in $\alpha\approx -0.588$. The curves are drawn for ten values of $n$ in a geometric progression with ratio $5$, starting from $n=1000$. Both the logarithmic scaling and the cutoff are visible. Moreover, the horizontal distance between the curves is approximately $\frac{\log 5}{-\alpha}\approx 2.738$, as predicted by our theorem.}
    \label{fig:hitallrange}
\end{figure}

In general, to estimate the  total variation distance one needs a more careful approach, even when $\gamma$ is the law of a normal random variable. Indeed, one has to approximate the distance between mixtures of normal random variables, for which no explicit formulas are available. However, our Theorem \ref{maintheorem} guarantees that both the logarithmic scaling and the cutoff are extremely robust.

\subsection{Structure of the paper}
The rest of the paper is organized as follows. In Section \ref{sec2} we briefly review a statistical motivation behind the model and other related literature. In Section \ref{sec3} we analyze the projection onto the unit sphere of the walk that is obtained by averaging over the additive noise. Then, in Section \ref{sec4} we leverage the properties of this chain to obtain our main result on the convergence rate and cutoff.

\section{Some background}\label{sec2}
\subsection{A statistical motivation}
Our interest in the problem came from a certain Bayesian scheme introduced by de Finetti \cite{de1938condition}: a large population is splitted into $d$ groups, and binary experiments are performed in each one. Under the assumption that people in the same group are indistinguishable, a situation referred to as \emph{partial exchangeability}, an approximate version of the classical de Finetti's theorem for exchangeable arrays -- which becomes exact for an infinite population -- reduces the problem to the understanding of certain measures $\pi$ on $\mathbb [0,1]^d$. The $i$-th coordinate $x_i\in [0,1]$ in a sample from $\pi$ should be interpreted as the likelihood of the experiment being successful in the $i$-th category. \\ \\One way to model the problem goes as follows: given a connected network on $d$ vertices with no loops and with weights $c_{ij}$ together with $x^{*}\in\mathbb R^d$ and $g\in\mathbb R_{+}^d$, consider the quadratic form
\begin{align*}
    Q(x)&=\sum_{1\leq i\leq j\leq d}c_{ij}(x_i-x_j)^2+\sum_{1\leq i\leq d}g_i(x_i-x_i^*)^2.
\end{align*}
Here $Q(x)$ can be also reparameterized as follows:
\begin{align*}    
    Q(x)&=\sum_{i=1}^d\frac{1}{\sigma^2_i}\left[e_i\left(x_i-\hat x_i\right)^2+\left(1-e_i\right)\left(x_i-x_i^*\right)^2\right],
\end{align*}
where we have
\begin{align*}
    e_i:=\frac{\sum_{j\neq i}c_{ij}}{\sum_{j\neq i}c_{ij}+g_i},\quad  \sigma_i^2:=\frac{1}{\sum_{j\neq i}c_{ij}+g_i},
\end{align*}
while $\hat x_i:=\sum p_{ij}x_j$ is defined via the transition probabilities on the network (i.e., $p_{ij}:=\frac{c_{ij}}{\sum_j c_{ij}}$). Finally, define $\pi^{(n)}$ to be the truncated Gaussian measure
\begin{align*}
    \pi^{(n)}(x)\,\,\,\propto \,\,e^{-n\frac{Q(x)}{2}},
\end{align*}
on the unit cube, where $n$ is a large parameter. \\ \\
The case $g_i\equiv 0$ (equivalently, $e_i\equiv 1$) corresponds to a prior situation where only the discrepancies in the likelihoods $x_i-x_j$ in different groups are taken into account, which are weighted by the $c_{ij}$s in the network. For $n$ large, this is referred to as the \emph{almost exchangeable} case since a sample from $\pi$ will typically consists of a point where all coordinates $x_i$s are roughly equal -- i.e., the result of the experiment on a given person is mildly affected by his/her group. \\ \\
After (sufficient) data are collected from each group, central limit theorem considerations lead to an approximately Gaussian Bayesian factor. In the measure $Q$, this is represented by the coefficients $g_i$ becoming positive (equivalently, $e_i\in (0,1)$). We refer to the last chapter of \cite{ottolini2021birthday} or \cite{Diaconis2022} for more background.\\ \\
To overcome numerical problems arising from the truncation \cite{Genz1992}, a first approach is to use rejection sampling. Alternatively, one can utilize a Gibbs sampler, since sampling from one-dimensional truncated normal distributions can be done efficiently \cite{Chopin2010}.
Standard concentration inequalities (see \cite{ottolini2021birthday}) show that the mixing time of the Gibbs sampler is only mildly affected by the truncation for $n$ large, as long as $e_i\in (0,1)$ for all $1\leq i\leq d$. Then, the problem boils down to the understanding of the auto-regressive model with the $Z_k$'s being standard normals. 
\\ \\
As a corollary of Theorem \ref{maintheorem}, we obtain that the mixing time for the Gibbs sampler associated to $\pi^{(n)}$ is of order $\log n$ as long as $e_i>0$. The case $e_i\equiv 1$ behaves rather differently and the mixing time becomes instead of order $n$ (see \cite{gerencser2020rates}).
\subsection{Related work}
There has been substantial work to understand the evolution of dynamics similar to the current setup. If we disregard the additive noise, we see that the starting point is closely related to the seminal result on random matrix products \cite{Furstenberg1960}.
\begin{proposition}[Fürstenberg-Kesten theorem, \cite{Furstenberg1960}]
\label{prp:F-K}
Let $(C_k)_{k=1}^\infty$ be a strictly stationary ergodic series of $d\times d$ matrices such that $\E \log^+ \|C_1\| < \infty$. Then the following limit exists almost surely:
$$
\lambda_1 = \lim_{k\to \infty}\frac 1k \log \|C_k C_{k-1}\cdots C_1\|
 = \lim_{k\to \infty}\frac 1k \E \log \|C_k C_{k-1}\cdots C_1\|.
$$
In this expression $\lambda_1<\infty$, but $\lambda_1=-\infty$ may occur.
\end{proposition}
The cited result is general in terms of applicability, having minimal constraints on the matrix series. In the current setting, however, we want to understand the evolution at a finite horizon rather than in an asymptotic manner.
\\ \\
Substantial work on the discrepancy from the above limit rate has also been carried out. For an i.i.d.\ series of invertible matrices, $\frac{1}{\sqrt{k}}(\log \|C_k C_{k-1}\cdots C_1\|-k\lambda_1)$ is asymptotically normal, as it was shown in \cite{PSMIR_1980___1_A4_0} and refined in \cite{Benoist2016} where the optimal moment conditions were determined, namely $\E\log^2\max(\|C_1\|,\|C_1^{-1}\|)<\infty$.
Similar results are available when other structural requirements are made, in particular for allowable matrices. A non-negative matrix is allowable, if all rows and columns contain strictly positive elements. Together with additional assumptions, a central limit theorem is shown to hold for stationary ergodic random products of such matrices (see \cite{Hennion1997}).
\\ \\
Observe that the set of matrices $A_i, 1\le i \le d$ currently studied are neither invertible nor allowable, as the $i$th column of $A_i$ has all zero entries, which suggests the specialized challenge. 
\\ \\
Moreover, our model \eqref{regressive} requires taking into account the additive term besides the linear map during the updates.
One can consider a setup of even wider generality, by randomly iterating maps in a complete separable metric space $(S,\rho)$. That is, define a Markov chain using a collection of maps $\{f_\theta \mid \theta\in\Theta\}$ by 
\begin{equation}
    \label{eq:D-F-MCdef}
X_0=x_0,\quad X_{k+1}=f_{\theta_{k+1}}(X_k),
\end{equation}
with i.i.d.\ indices $\theta_k$ according to a distribution $\mu$.
In this framework, stability can be ensured as follows.
\begin{proposition}[\cite{Diaconis1999}, Theorem 1.1.]
In the above setup, assume for all $\theta\in \Theta$ that $f_\theta$ is Lipschitz with Lipschitz constant $K_\theta$. We further assume $\int_\Theta K_\theta \mu(d\theta)<\infty$, $\int_\Theta \log K_\theta \mu(d\theta)<0$, and for some $x_0\in S$, $\int_\Theta \rho(x_0,f_\theta(x_0)) \mu(d\theta)<\infty$.

Then the Markov chain in \eqref{eq:D-F-MCdef} has a unique stationary distribution, and exponential convergence occurs in the Prokhorov metric. Here, the rate is bounded away from 0 uniformly in $x_0$.

Moreover, the backward recursion $f_{\theta_1}\circ f_{\theta_2}\circ \ldots \circ f_{\theta_k}(x_0)$ converges almost surely.
\end{proposition}
This tool is powerful for its generality -- several Markov chains can be cast in this language -- and it highlights the \emph{contracting in average} condition. While this can be relaxed in the affine case (see Theorem $2.1$ in \cite{Diaconis1999}) to include our case, it still does not capture exactly the rate as Proposition \ref{prp:F-K}.
\\ \\
Therefore, our setup and claim fall outside the regime of the important works reviewed above.

\section{A random walk on the sphere, in the positive cone}\label{sec3}
As already hinted at, the convergence rate of the chain $X^{(n)}_k$, as defined in \eqref{regressive}, is essentially determined by the concentration properties of $Y_k$ defined via
\begin{equation}\label{walkonsphere}
    Y_0=\frac{X_0}{\|X_0\|}, \quad Y_k=\frac{A_{I_k}Y_{k-1}}{\|A_{I_k}Y_{k-1}\|},
\end{equation}
Notice that this is well defined as long as $X_0$ has positive coordinates, since the only effect of $A_{I_k}$ (defined in \eqref{contractionpart}) is to replace one coordinate with a damped weighted average of the others.
It is worth noting that $Y_k$ is simply the chain obtained by averaging over the additive noise, and then normalized to lie on the unit sphere, i.e., 
\begin{equation}\label{Yonthesphere}
    Y_k=\frac{\mathbb E_Z[X^{(n)}_k]}{\|\mathbb E_Z[X^{(n)}_k]\|},
\end{equation}
where $\mathbb E_Z$ denotes the expectation with respect to  $Z_1,\ldots, Z_k$. Notice that there is no dependence on $n$ because of the linearity of expectation and $\mathbb E[Z_1]=0$. \\ \\
Our first goal is to show a uniform bound on the ratio between coordinates in $Y_k$. We start by introducing some notation: for $y, y'\in\mathbb R^d$ write $y>y'$ or $y\geq y'$ if the same type of inequality holds coordinatewise. For $\delta\geq 0$, let $\mathbf 1_{\delta}$ be the vector with all components identically equal to $\delta$. Then, let  
\begin{align*}
\mathcal S_{\delta}:=\{y\in\mathbb R^d \,|\, \|y\|=1, y>\mathbf 1_{\delta} \}.
\end{align*}
We write $\mathcal S$ for $\mathcal S_0$. Notice that our assumption on $X_0$ implies that $Y_k\in \mathcal S$ for all $k$. When each $e_i$ in the definition \eqref{contractionpart} is equal to one, then convexity entails that $\frac{\min Y_k}{\max Y_k}$ remain bounded away from zero. The first step is to generalize this to the case $e_i<1$, namely by showing that regardless of the choices of the updates, $Y_k\in\mathcal S_{\delta}$ for some $\delta>0$ that does not depend on $k$.
\begin{lemma}\label{avoidancelemma}
Let $\theta:=\min_{i\sim j}e_ip_{ij}>0$, where $i\sim j$ denotes a pair for which $p_{ij}>0$. Then, for all choices of the updates, one has the bound
\begin{equation}\label{claimforlemma21}
\frac{\min Y_k}{\max Y_k}\geq \frac{\min Y_0}{\max Y_0}\theta^{d-1}.
\end{equation}
\end{lemma}
\begin{proof}
We start by observing that, since the statement we aim to prove is scale-invariant, we can study the chain $Y_k$ where we neglect the normalization in \eqref{walkonsphere}. For convenience, we still denote it by $Y_k$. \\ \\
Moreover, if for another chain $\tilde Y_k$ we have $Y_0\geq \tilde Y_0$ and the same updates are used for both chains, then $Y_k\geq \tilde Y_k$ for all $k$. Together with the observation
\begin{align*}
   \mathbf 1_{\min Y_0}\leq  Y_0\leq \mathbf 1_{\max Y_0},
\end{align*}
it suffices to prove the statement when $Y_0=\mathbf 1_{1}$, for which the right side of \eqref{claimforlemma21} is just $\theta^{d-1}$. We now proceed by induction on $k$ as follows.\\ \\
For all $1\leq i\leq d$ we have 
\begin{align*}
    1=(Y_0)_i\geq (e_i\widehat Y_0)_i=e_i.
\end{align*}
We will now prove that the inequality in the middle holds for $Y_1$ as well, regardless of the choice of the updated index $I$. In fact, there are three possibilities:
\begin{itemize}
    \item If $i\neq I$, $i\not\sim I$, then our assumption leads to
    \begin{align*}
        (Y_1)_i=(Y_0)_i\geq (e_i\widehat Y_0)_i=(e_i\hat Y_1)_i.
    \end{align*}
    since $Y_0$ and $Y_1$ coincide everywhere except on the $I$th coordinate.
    \item If $i=I$, then
    \begin{align*}
        (Y_1)_i=(e_i\widehat Y_0)_i=(e_i\widehat Y_1)_i.
    \end{align*}
   \item If $i\sim I$, then
 \begin{align*}
     (Y_1)_i=(Y_0)_i\geq (e_i\hat Y_0)_i=(e_i\hat Y_1)_i+e_ip_{Ii}\left(-(e_I\hat Y_0)_I+(Y_0)_I\right)\geq (e_i\hat Y_1)_i.
 \end{align*}
 \end{itemize}
 Iterating the argument above $k$ times for the subsequent updates, we obtain that for all choices of the updates and for all coordinates $1\leq i\leq d$, 
 \begin{align*}
     (Y_k)_i\geq (e_i\hat Y_k)_i.
 \end{align*}
For any choice of $j\sim i$, owing to the definition \eqref{averagingdef} and our choice of $\theta$ we can bound 
\begin{align*}
(Y_k)_i\geq \theta (Y_{k})_j.
\end{align*}
Consider now connecting the extremal coordinates $\text{argmin } Y_k=i_1\sim i_2\ldots i_{t-1}\sim i_t=\text{argmax } Y_k$ using a shortest path. Then, iterating the inequality above we obtain
\begin{align*}
\min Y_k \geq \theta^{t-1}\max Y_k, 
\end{align*}
and we conclude since $t\leq d$.
\end{proof}
\begin{remark}
Here is a simple geometric interpretation of the proof. Each of the $A_i$'s projects a point -- in a non-orthogonal fashion -- onto some hyperplane $H_i$. We exploit that the connected component of $\mathbb R^d\,\setminus \bigcup_{i=1}^d H_i$ containing $\mathbf 1_{1}$ is invariant under our dynamic.
\end{remark}
\begin{remark}
The inequality above is sharp when the underlying network structure is a line path of length $d$. 
\end{remark}
Armed with this lemma, our next step is to show that the law of $Y_k$ converges to a unique measure, independently of the starting position $Y_0\in\mathcal S$. 
\subsection{Weak contraction in the Hilbert metric}
Consider the Hilbert metric $h$ on $\mathcal S$, given by
\begin{align*}
    h(y, y'):=\log\left( \frac{\max \frac{y_i}{y'_i}}{\min \frac{y_i}{y'_i}}\right).
\end{align*}
Consider also the corresponding Wasserstein metric induced on Borel probability measures on $\mathcal S$
\begin{align*}
    W(\mu, \nu)=\inf_{Y\sim \mu, Y'\sim \nu}\mathbb E\left(h(Y, Y')\right),
\end{align*}
where the infimum is taken over all couplings $Y\sim \mu, Y'\sim \nu$.\\ \\
Let us highlight a few properties of these metrics: the space $\mathcal S_{\delta}$ is compact for all $\delta>0$, but not for $\delta=0$, when equipped with the metric $h$. Moreover, convergence in the metric $W$ on the space of Borel measures on $\mathcal S_{\delta}$, $\delta>0$, is tantamount weak convergence, owing to the boundedness of the metric. In particular, Prokhorov's theorem guarantees that the compactness property is inherited by the Wasserstein space.  \\ \\
This allows us to prove the following result. 
\begin{lemma}\label{invariantunique}
There exists a unique limit $\nu$ for the law of $Y_k$, which is independent of the choice of $Y_0\in\mathcal S$.
\end{lemma}
\begin{proof}
Let $Y_0\in\mathcal S$ be arbitrary and let $\delta$ be small enough so that $Y_k\in\mathcal S_{\delta}$ for all $k$, which we can ensure owing to Lemma \ref{avoidancelemma}. A compactness argument yields immediately the existence of a limiting measure $\nu$ up to subsequences.\\ \\
In order to show uniqueness, regardless of the initial condition, it suffices to show that for \emph{any} pair of measures $\nu\neq \nu'$ on $\mathcal S_{\delta}$ we have a weak contraction between $\nu$ and $\nu'$ after $d$ steps of the Markov chain, i.e.,
\begin{align}
\label{eq:weakWcontraction}
   W(\nu_d, \nu'_d)<W(\nu, \nu').
\end{align}
Here, $\nu_d, \nu'_d$ denote the laws of $Y_d, Y'_d$ with $Y_0, Y_0'$ being distributed according to $\nu, \nu'$.
Indeed, if both $\nu\neq \nu'$ were stationary measures, then we would obtain
\begin{align*}
   W(\nu, \nu')=W(\nu_d, \nu'_d)<W(\nu, \nu'),
\end{align*}
which is a contradiction, and thus $\nu=\nu'$. \\ \\
Owing to the convexity of the Wasserstein metric, it suffices to show the bound \eqref{eq:weakWcontraction} for $\nu$ and $\nu'$ being delta masses at some $Y_0$ and $Y'_0$, in which case the right side becomes $h(Y_0,Y'_0)$ for some $Y_0$ and $Y'_0$ in $\mathcal S_{\delta}$ for some $\delta$.\\ \\
Consider now the coupling where the same coordinates are updated for both $Y_0$ and $Y'_0$. After one step, the ratio $(Y_1)_i/(Y'_1)_i$ either remains the same (if coordinate $i$ is not selected) or is equal to the ratio of a weighted average of all other coordinates. In both cases, we have
\begin{align*}
    \min \frac{Y_0}{Y'_0}\leq \frac{(Y_1)_i}{(Y'_1)_i}\leq \max \frac{Y_0}{Y'_0}.
\end{align*}
Iterating, we obtain for all choices of indices
\begin{align*}
    \min \frac{Y_0}{Y'_0}\leq \min \frac{Y_d}{Y'_d}\leq \max \frac{Y_d}{Y'_d}\leq  \max \frac{Y_0}{Y'_0}, 
\end{align*}
which implies $h(Y_d, Y_d')\leq h(Y_0, Y'_0)$ for all choices of the updated indices. \\ \\Moreover, there exists a selection of indices $I_1,\ldots, I_d$ for which the rightmost inequality is strict. To show this, let $\mathcal I$ be the set of indices where the maximum $\max \frac{Y_0}{Y'_0}$ is achieved. Notice that $|\mathcal I|<d$ (by the assumption that $Y_0$ and $Y'_0$ are distinct), and that at least one element $I$ of $\mathcal I$ is connected to an element of $\mathcal I^c$ (since the network is connected). Therefore, if we start by selecting $I$, $\frac{(Y_1)_I}{(Y'_1)_I}<\frac{(Y_0)_I}{(Y'_0)_I}$ and thus the cardinality of $\mathcal I$ drops by one. Iterating this at most $d$ times, we obtain the conclusion.\\ \\
Let us denote the event of a specific such index series occurring by $A$, and the corresponding instance of the Markov chain after $d$ steps by $Y_d(A),Y'_d(A)$. Similarly, the event for any other index series is denoted by $A^c$, and the corresponding conditional version of the Markov chain by $Y_d(A^c),Y'_d(A^c)$. Our previous observations entail
\begin{align*}
\quad h(Y_d(A), Y'_d(A))<h(Y_0, Y'_0),\quad h(Y_d(A^c), Y'_d(A^c))\leq h(Y_0, Y'_0),
\end{align*}
so that we obtain 
\begin{align*}
    W(\nu_d, \nu'_d)&\leq \left(1-\frac{1}{d^d}\right)h(Y_d(A^c),Y'_d(A^c))+\frac{1}{d^d}h(Y_d(A), Y'_d(A))\\&<h(Y_0,Y'_0)\\&=W(\nu, \nu')
\end{align*}
as desired.
\end{proof}
\begin{remark}\label{statsupport}
The choice of $Y_0$ with identical coordinates shows that the unique stationary measure has support contained in $\mathcal S_{\theta^{d-1}}$, for $\theta$ defined in Lemma \ref{avoidancelemma}.
\end{remark}
\subsection{Concentration inequalities}
Let $\overline Y_0$ denote a random variable distributed according to $\nu$, the unique stationary measure given by Lemma \ref{invariantunique}. As observed in Remark \ref{statsupport}, we obtain that $\overline Y_0\in S_{\theta^{d-1}}$ with probability one. In particular, this shows that
\begin{equation}\label{alpha}
    \alpha:=\mathbb E[\log \|A_{I}\overline Y_0\|]\in (-\infty,0),
\end{equation}
where the expectation is taken over $\overline Y_0\sim \nu$ and $I$ uniformly distributed in $\{1,\ldots d\}$. Since $\nu$ is stationary we deduce
\begin{equation}\label{stationaryy}
    \overline Y_1=\frac{A_I \overline Y_0}{\|A_I \overline Y\|}\stackrel{d}{=}\overline Y_0.
\end{equation}
and more generally $\overline Y_k\stackrel{d}{=}\overline Y_0$. \\ \\ 
We now need to connect properties of the chain $Y_k$ back to properties of $X^{(n)}_k$. Owing to the invariance of total variation under re-scaling, in the proof of Theorem \ref{maintheorem} we can restrict our attention to a fixed Markov chain dynamics $X_k:=X^{(1)}_k$ where the dependence on $n$ is only through the initial conditions $X^{(n)}_0$ (as a matter of fact, we neglect the dependence on $n$ of the whole chain with a slight abuse of notation). The assumption of Theorem \ref{maintheorem} guarantees that the ratio between the minimum and the maximum coordinate remains bounded away from zero uniformly for given $a_-, a_+$, and such that $\|X^{(n)}_0\|=\Theta_{a_{-},a_{+}}(n)$ (here, the notation $f^{(n)}=\Theta(g^{(n)})$ means that $f^{(n)}/g^{(n)}$ remains bounded away from zero and infinity). 
\\ \\
If $X'^{(n)}_0=\|X^{(n)}_0\|\overline Y_0$, combining \eqref{stationaryy} with \eqref{Yonthesphere} we obtain
\begin{equation}\label{propertyalpha}
    \mathbb E[\log\|\mathbb E_{Z}[X'^{(n)}_k]\|]-\log \|X_0^{(n)}\|=\sum_{j=1}^k\mathbb E[\log \|A_{I_j} \overline Y_j\|]=k\mathbb E[\log \|A_I\overline Y_0\|]=k\alpha.
\end{equation}
Armed with this, we can prove the following.
\begin{lemma}\label{concentration}
Let $Y_0\in \mathcal S_{\delta}$ and $X_0^{(n)}=\|X_0^{(n)}\|Y_0$ with $\|X_0^{(n)}\|=\Theta_{a_-, a_+}(n)$. Then, there exists a constant $\gamma=\gamma(a_-,a_+)>0$, independent of $k$ and $n$, such that
\begin{align*}
    \mathbb P\left[\left|\log\left(\frac{\|\mathbb E_{Z}[X_k]\|}{\|X_0^{(n)}\|}\right)-k\alpha\right|\geq  t\sqrt k \right]\leq 2e^{-\gamma t^2}
\end{align*}
for all $t>0$ and $k$ large enough.
\end{lemma}
\begin{proof}
In what follows, the symbol $\lesssim$ denotes an inequality up to a constant independent of $k, n$ and the choice of indices. Construct $X'_k$ from $X_k$ by re-sampling the $s$th update for some $1\leq s\leq k$. Then
\begin{align*}
\frac{\|\mathbb E_{Z_1,\ldots, Z_k}[X_k]\|}{\|\mathbb E_{Z_{1},\ldots, Z_k}[X'_k]\|}=\frac{\|\mathbb E_{Z_{s+1},\ldots, Z_k}[A_{I_{s}}\tilde Y_0]\|}{\|\mathbb E_{Z_{s+1},\ldots, Z_k}[A_{I'_{s}}\tilde Y_0]\|}
\end{align*}
for some $\tilde Y_0\in \mathcal S_{\delta'}$, where $\delta'$ depends on $\delta$ and $\theta$ only owing to Lemma \ref{avoidancelemma}. Since for all indices $I$
\begin{align*}
    \mathbf 1_1 \lesssim A_{I}\tilde Y_0 \lesssim \mathbf 1_1,
\end{align*}
we can use Lemma \ref{avoidancelemma} applied to $\mathbf 1_{1}$ to deduce
\begin{align*}
1 \lesssim  \frac{\|\mathbb E_{Z_1,\ldots, Z_k}[X_k]\|}{\|\mathbb E_{Z_{1},\ldots, Z_k}[X'_k]\|}\lesssim 1
\end{align*}
or, equivalently,
\begin{align*}
    1 \lesssim \left|\log \|\mathbb E_{Z_1,\ldots, Z_k}[X_k]\|-\log\|\mathbb E_{Z_1,\ldots, Z_k}[X'_k]\|\right|\lesssim 1. 
\end{align*}
Using the bounded difference inequality \cite{McDiarmid1989}, we obtain the claim where $k\alpha$ is replaced by $\mathbb E\log\left[\frac{\|\mathbb E_{Z}[X_k]\|}{\|X^{(n)}_0\|}\right]$. Therefore, it suffices to show that 
\begin{align*}
   \left| \mathbb E\log \|\mathbb E_{Z_1,\ldots, Z_k}[X_k/\|X^{(n)}_0\|]\|-k\alpha\right|\lesssim 1,
\end{align*}
for then the claim follows by possibly decreasing $\gamma$ and taking $k$ large enough. Using \eqref{propertyalpha}, we obtain
\begin{align*}
    \mathbb E\log \|\mathbb E_{Z_1,\ldots, Z_k}[X_k/\|X^{(n)}_0\|]\|-k\alpha=\mathbb E \left[\log \frac{\|\mathbb E_{Z_1,\ldots, Z_k}[X_k]\|}{\|\mathbb E_{Z_1,\ldots, Z_k}[X'_k]\|}\right].
\end{align*}
where $X'^{(n)}_0=\|X^{(n)}_0\|\overline Y_0$ with $\overline Y_0$ distributed according to the stationary distribution $\nu$, and we use the same updates on both $X_k$ and $X'_k$. Since $\overline Y_0\in S_{\theta^{d-1}}$ with probability one (see Remark \ref{statsupport}), the claim then follows applying once more Lemma \ref{avoidancelemma} as before.
\end{proof}
\section{Proof of the main result}\label{sec4}
As hinted at in the previous section, we can restrict our attention to the Markov chain $X_k$, with a varying initial condition $X_0^{(n)}$, and $\overline X=\overline X^{(1)}$, the latter being given by \eqref{stationary}. Notice that $\overline X=O(1)$ with high probability. We will denote the corresponding laws by $\pi_k$ and $\pi$. We start by observing that for any choice of $X^{(n)}_0$ one has
\begin{align*}
    X_k-A_{I_k}\ldots A_{I_1}X^{(n)}_0&=A_{I_k}\ldots A_{I_1}b_{I_1}(Z_{1})+ A_{I_k}\ldots A_{I_2}b_{I_2}(Z_{2})+ \ldots+b_{I_k}(Z_{k})\\&\stackrel{d}{=} A_{I_1}\ldots A_{I_k}b_{I_k}(Z_{k})+ A_{I_1}\ldots A_{I_{k-1}}b_{I_{k-1}}(Z_{k-1})+ \ldots + b_{I_1}(Z_{1}),
\end{align*}
where we used exchangeability of the sequences $I_1,\ldots, I_k$ and $Z_1,\ldots, Z_k$. This entails
\begin{equation}\label{xminusinitial}
     X_k-A_{I_k}\ldots A_{I_1}X^{(n)}_0\stackrel{d}{\to}\overline X.
\end{equation}
We are now ready to prove our main result.
\begin{proof}[Proof of Theorem \ref{maintheorem}]
We start proving the first claim, namely 
\begin{align*}
  \lim_{\beta\rightarrow -\infty}\lim_{n\rightarrow +\infty}d_{tv}(\pi, \pi_k)=1.
\end{align*} By definition of total variation distance, it is enough to show that for all $\epsilon>0$ there exists $\beta\in\mathbb R$ and $R>0$ such that $|\mathbb P(\overline X\in B_R)-\mathbb P(X_k\in B_R)|\geq 1-\epsilon$ for all $k=k(n,\beta)$ (given in \eqref{choiceofk}) with $n$ large enough. Here, $B_R$ denotes the ball centered at the origin with radius $R$. \\ \\
Fix $\epsilon>0$, and pick $R$ large enough so that for all $k$ sufficiently large
\begin{align*}
    \mathbb P(\overline X\in B_R)\geq 1-\frac{\epsilon}{3}, \quad \mathbb P(X_k-A_{I_k}\ldots A_{I_1}X^{(n)}_0\in B_{R})\geq 1-\frac{\epsilon}{3}.
\end{align*}
This is possible owing to \eqref{xminusinitial}. A union bound leads to
\begin{align*}
    \mathbb P(X_k\in B_R)&\leq \mathbb P(X_k-A_{I_k}\ldots A_{I_1}X^{(n)}_0\not\in B_{R})+\mathbb P(A_{I_k}\ldots A_{I_1}X^{(n)}_0\in B_{2R})\\&\leq \frac{\epsilon}{3}+\mathbb P(A_{I_k}\ldots A_{I_1}X^{(n)}_0\in B_{2R}).
\end{align*}
Therefore, we have
\begin{align*}
    \mathbb P(\overline X\in B_R)-\mathbb P(X_k\in B_R)\geq 1-\frac{2\epsilon}{3}-P(A_{I_k}\ldots A_{I_1}X^{(n)}_0\in B_{2R}),
\end{align*}
so that we obtain the claim provided that 
\begin{align*}
    P(A_{I_k}\ldots A_{I_1}X^{(n)}_0\in B_{2R})\leq \frac{\epsilon}{3}
\end{align*}
for $k$ as in \eqref{choiceofk} with $\|X^{(n)}_0\|=\Theta_{a_-, a_+}(n)$ large and $\beta$ small enough. Since $Z$ has mean zero, we have
\begin{align*}
    \mathbb E_{Z}[X_k]=A_{I_k}\ldots A_{I_1}X^{(n)}_0,
\end{align*}
and passing to logarithms we need to bound
\begin{align*}
    \mathbb P\left(\log \|\mathbb E_{Z}[X_k]\| \leq \log(2R)\right).
\end{align*}
Thanks to Lemma \ref{concentration}, we know that for all $t\geq 0$ and all $k$ sufficiently large
\begin{align*}
    \mathbb P\left(\log \|\mathbb E_{Z}[X_k]\|\leq \log n+\Theta_{a_-, a_+}(1)+k\alpha-t\sqrt k\right)\leq 2e^{-\gamma t^2}.
\end{align*}
In order to conclude, take $t$ large enough so that the right side is smaller than $\frac{\epsilon}{3}$. Then, for $k$ as in \eqref{choiceofk} we have
\begin{align*}
    \log n+k\alpha-t\sqrt k+\Theta_{a_-, a_+}(1)=\sqrt{\log n}\left(-\beta-\frac{t}{\sqrt{-\alpha}}+o(1)\right)\geq \log (2R)
\end{align*}
for $\beta$ negative and with a large enough absolute value. Therefore, 
\begin{align*}
    \mathbb P\left(\log \|\mathbb E_{Z}[X_k]\| \leq \log (2R)\right)\leq \frac{\epsilon}{3}
\end{align*}
for all $k=k(n,\beta)$ with $n$ sufficiently large, as desired.  \\ \\
We now move to the second claim, namely
\begin{align*}
\lim_{\beta\rightarrow +\infty}\lim_{n\rightarrow +\infty}d_{tv}(\pi, \pi_k)=0.
\end{align*}
Consider an arbitrary $X^{(n)}_0\in \mathcal S_{\delta}$ for some $\delta>0$ with $\|X^{(n)}_0\|=\Theta_{a_-, a_+}(n)$, 
and let $X'_0$ be distributed according to the stationary distribution (notice that $\|X'_0\|=O(1)$). Owing to the coupling interpretation of total variation distance we need to show that, for all $\epsilon>0$, one can construct a coupling between $X_k$ and $X'_k$ such that
\begin{align*}
    \mathbb P(X_k\neq X_k')\leq \epsilon
\end{align*}
for $k=k(n,\beta)$ as in \eqref{choiceofk} with $\beta$ large enough and all $n$ sufficiently large. \\ \\
Let $T$ denote the first time that all coordinates have been selected at least once. On $\{T\geq k\}$, let the two chains $X_k$ and $X_k'$ run independently. Conditioned on $\{T\leq k\}$, let $1\leq k_i\leq k$ be the last time that coordinate $i$ is selected, $i\in\{1,\ldots, d\}$. Consider a coupling between $X_k$ and $X'_k$ with the same choice of coordinate updates, and using the same additive noise except at the times $k_i$. Without loss of generality, assume that $k_1 < k_2 < \ldots < k_d=k$. Then we can write
\begin{align*}
    X_k-X'_k=A_{I_k}\ldots A_{I_1}(X^{(n)}_0-X'_0)+G(Z_{k_1}-Z'_{k_1},\ldots, Z_{k_d}-Z'_{k_d}),
\end{align*}
where $G=G_{I_1,\ldots, I_k}$ is the linear map that sends $z\in\mathbb R^d$ to 
\begin{align*}
    G(z)=\sum_{i=1}^dA_{I_k}\ldots A_{I_{k_i+1}}b_{I_{k_i}}(z).
\end{align*}
Notice that the last summand reduces to $(0,\ldots,0, \sigma_{d}z_d)$, and in general the $i$th summand is a vector with the first $i-1$ entries being zero owing to \eqref{contractionpart} and the assumption $k_1 < \ldots < k_d$. In particular, the matrix $G$ is lower triangular with $\sigma_i$s on the diagonal, so that $G$ is invertible and its inverse has a uniformly bounded norm (with respect to $k$ and the choice of the indices).
\\ \\Moreover, for all choices of $r>0$ we have
\begin{align*}
    d_{TV}(\pi_k, \pi)&\leq  \mathbb P(X_k\neq X'_k)&\\ &\leq \mathbb P(T>k)+\mathbb P(\|A_{I_k}\ldots A_{I_1}(X^{(n)}_0-X'_0)\|\geq r)\\&\enspace+\sup_{\|s\|\leq r, I_1,\ldots, I_k}\mathbb P(Z\neq Z'+G^{-1}(s)).
\end{align*}
Here, $Z$ and $Z'$ are vectors in $\mathbb R^d$ with i.i.d. components distributed according to $\gamma$. The first terms is smaller than $\epsilon/3$ for $k$ large, owing to a classical coupon collector argument. As for the last term, we can couple $Z$ and $Z'$ optimally so that (here we identify $\gamma$ with its density) 
\begin{align*}
    \mathbb P(Z-Z'\neq G^{-1}(s))&\leq \frac{1}{2}\int_{\mathbb R^d}|\gamma(z_1)\ldots \gamma(z_d)-\gamma(z_1+G^{-1}_1(s))\ldots \gamma(z_d+G^{-1}_d(s))|dz \\ & \leq  \frac{d}{2}\sup_{i\in\{1,\ldots, d\}} \int_{\mathbb R}|\gamma(z)-\gamma(z+G_i^{-1}(s))|dz\\&\leq 
    \epsilon/3
\end{align*}
provided that $s\leq r$ with $r$ small enough, thanks to the uniform control on the inverse of $G$ and to the continuity of the translation operator on integrable functions. 
As for the second term, a union bounds yields
\begin{align*}
\mathbb P(\|A_{I_k}\ldots A_{I_1}(X^{(n)}_0-X'_0)\|\geq r)&\leq \mathbb P(\|A_{I_k}\ldots A_{I_1}X'_0\|\geq \frac{r}{2})\\&\enspace+\mathbb P(\|A_{I_k}\ldots A_{I_1}X^{(n)}_0)\|\geq \frac{r}{2}).
\end{align*}
Since $\|X_0'\|=O(1)$ and $\|X^{(n)}_0\|=\Theta_{a_-, a_+}(n)$ with $X^{(n)}_0\in S_{\delta}$ for some $\delta>0$, it is enough to show that the second term is smaller than $\epsilon/6$. On the other hand, using Lemma \ref{concentration} and following the very same approach of the proof of the first claim, we obtain that the second term is smaller than $\epsilon/6$ for $k=k(n,\beta)$ as in \eqref{choiceofk} with $\beta$ large and $n$ sufficiently large. \\ \\
Altogether, this implies the main result.
\end{proof}

\section*{Acknowledgment}
We warmly thank Persi Diaconis for suggesting the problem being studied, and for his constant help and support. B. Gerencs\'er was supported by NRDI (National Research, Development and Innovation Office) grant KKP 137490 and by the J\'anos Bolyai Research Scholarship of the Hungarian Academy of Sciences.

\bibliographystyle{abbrv}
\bibliography{Bibliography.bib}

\def\cprime{$'$}
\begin{thebibliography}{10}

\bibitem{Benoist2016}
Y.~Benoist and J.-F. Quint.
\newblock Central limit theorem for linear groups.
\newblock {\em The Annals of Probability}, 44(2), Mar. 2016.

\bibitem{Chopin2010}
N.~Chopin.
\newblock Fast simulation of truncated {G}aussian distributions.
\newblock {\em Statistics and Computing}, 21(2):275--288, 2010.

\bibitem{de1938condition}
B.~de~Finetti.
\newblock {\em Sur la condition d'" Equivalence partielle."}.
\newblock Actualities Scientifiques et Industrielles, 1938.

\bibitem{Diaconis1996}
P.~Diaconis.
\newblock The cutoff phenomenon in finite {M}arkov chains.
\newblock {\em Proceedings of the National Academy of Sciences},
  93(4):1659--1664, 1996.

\bibitem{Diaconis2022}
P.~Diaconis.
\newblock Approximate exchangeability and de {F}inetti priors in 2022.
\newblock {\em Scandinavian Journal of Statistics}, Oct. 2022.

\bibitem{Diaconis1999}
P.~Diaconis and D.~Freedman.
\newblock Iterated random functions.
\newblock {\em {SIAM} Review}, 41(1):45--76, Jan. 1999.

\bibitem{Furstenberg1960}
H.~Furstenberg and H.~Kesten.
\newblock Products of random matrices.
\newblock {\em The Annals of Mathematical Statistics}, 31(2):457--469, June
  1960.

\bibitem{Genz1992}
A.~Genz.
\newblock Numerical {C}omputation of {M}ultivariate {N}ormal {P}robabilities.
\newblock {\em Journal of Computational and Graphical Statistics},
  1(2):141--149, 1992.

\bibitem{gerencser2020rates}
B.~Gerencs{\'e}r and A.~Ottolini.
\newblock Rates of convergence for {G}ibbs sampling in the analysis of almost
  exchangeable data.
\newblock {\em arXiv preprint arXiv:2010.15539}, 2020.

\bibitem{1995}
W.~Gilks, S.~Richardson, and D.~S. (eds.).
\newblock {\em Markov {C}hain {M}onte {C}arlo in {P}ractice}.
\newblock Chapman and Hall/{CRC}, 1995.

\bibitem{Hennion1997}
H.~Hennion.
\newblock Limit theorems for products of positive random matrices.
\newblock {\em The Annals of Probability}, 25(4), Oct. 1997.

\bibitem{PSMIR_1980___1_A4_0}
E.~Le~Page.
\newblock Th\'eor\`emes limites pour les produits de matrices al\'eatoires.
\newblock {\em Publications math\'ematiques et informatique de Rennes}, 1980.

\bibitem{LevinPeresWilmer2006}
D.~A. Levin, Y.~Peres, and E.~L. Wilmer.
\newblock {\em {Markov chains and mixing times}}.
\newblock American Mathematical Society, 2006.

\bibitem{McDiarmid1989}
C.~McDiarmid.
\newblock On the method of bounded differences.
\newblock In {\em Surveys in Combinatorics, 1989}, pages 148--188. Cambridge
  University Press, Aug. 1989.

\bibitem{ottolini2021birthday}
A.~Ottolini.
\newblock {\em Birthday Problems and Rates of Convergence for Gibbs Sampling}.
\newblock Stanford University, 2021.

\end{thebibliography}
\end{document}